\documentclass[11pt,reqno,tbtags]{amsart}

\usepackage{amsthm,amssymb,amsfonts,amsmath}
\usepackage{amscd} %Commutative diagrams. 
\usepackage{mathrsfs}
\usepackage[numbers, sort&compress]{natbib}
\usepackage{subfigure, graphicx}
\usepackage{vmargin}
\usepackage{setspace}
\usepackage{paralist}
\usepackage{hyperref}
\usepackage{color}
\usepackage[normalem]{ulem}
\usepackage{stmaryrd} %for \bbr
\usepackage[refpage,noprefix,intoc]{nomencl} %for list of notation.

\newcommand{\E}[1]{{\mathbf E}\left[#1\right]}
\newcommand{\e}{{\mathbf E}}

\newcommand{\p}[1]{{\mathbf P}\left\{#1\right\}}

\newtheorem{thm}{Theorem}
\newtheorem{lem}[thm]{Lemma}
\newtheorem{prop}[thm]{Proposition}
\newtheorem{cor}[thm]{Corollary}
\newtheorem{dfn}[thm]{Definition}

\numberwithin{equation}{section} %Numbers equations by section
\numberwithin{thm}{section}

\providecommand{\ora}[1]{}
\renewcommand{\ora}[1]{\overrightarrow{#1}}

\hypersetup{
    bookmarks=true,         % show bookmarks bar?
    unicode=false,          % non-Latin characters in Acrobat’s bookmarks
    pdftoolbar=true,        % show Acrobat’s toolbar?
    pdfmenubar=true,        % show Acrobat’s menu?
    pdffitwindow=true,      % page fit to window when opened
    pdftitle={My title},    % title
    pdfauthor={Author},     % author
    pdfsubject={Subject},   % subject of the document
    pdfnewwindow=true,      % links in new window
    pdfkeywords={keywords}, % list of keywords
    colorlinks=true,       % false: boxed links; true: colored links
    linkcolor=blue,          % color of internal links
    citecolor=blue,        % color of links to bibliography
    filecolor=blue,      % color of file links
    urlcolor=blue           % color of external links
}

\newcommand\urladdrx[1]{{\urladdr{\def~{{\tiny$\sim$}}#1}}}
\begingroup
  \divide\count255 by 60
  \multiply\count255 by -60
  \advance\count255 by \time
  \ifnum \count255 < 10 \xdef\oclock{\the\count1:0\the\count255}
  \else\xdef\oclock{\the\count1:\the\count255}\fi
\endgroup

\makenomenclature							

\newcommand\cF{\mathcal F}

\newcommand\cT{{\mathcal T}}

\DeclareRobustCommand{\SkipTocEntry}[5]{}

\usepackage{bbm}
\numberwithin{figure}{section} % make sure for article class, figure is numbered using format Fig 1.1, rather than Fig 1  This uses amsmath package

\begin{document}

\title{Small trees in supercritical random forests} 
\author{Tao Lei}
\address{Department of Mathematics and Statistics, McGill University, 805 Sherbrooke Street West, 
		Montr\'eal, Qu\'ebec, H3A 0B9, Canada}
%\email{louigi.addario@mcgill.ca}
\email{tao.lei@mail.mcgill.ca}
\date{October 5, 2017} %; revised ...
%\urladdrx{http://www.problab.ca/louigi/}
\urladdrx{http://www.math.mcgill.ca/~tlei/}


%{60C05 (68P10,68W40)} %%{Primary: <subject>; Secondary: <subject>}
\begin{abstract} 
We study the scaling limit of random forest with prescribed degree sequence in the regime that the largest tree consists of all but a vanishing fraction of nodes. We give a description of the limit of the forest consisting of the small trees, by relating plane forest to marked cyclic forest and its corresponding lattice path. 
\end{abstract}

\maketitle

%\tableofcontents
%{\textcolor{red}{Abstract! Correct date!}}

%%%%%%%%%%%%%%%%%%%%%%%%%%%%
% INTRODUCTION
%%%%%%%%%%%%%%%%%%%%%%%%%%%%
\section{\bf Introduction}\label{sec:intro} 
A plane tree is a finite rooted tree in which the children of each node are ordered. A plane forest is a finite sequence of plane trees. 

Fix a plane tree $T$ and a plane forest $F=(T_1,\ldots,T_c)$. The node set of $F$ is $v(F) = \bigsqcup_{i \le c} v(T_i)$, where $v(T_i)$ is the node set of $T_i$ and where $\bigsqcup$ denotes disjoint union. For a node $v \in v(T)$, by the {\em degree} of $v$ we mean the number of children of $v$ in $T$. We denote this quantity $k_T(v)$.  The degree of $v \in v(F)$, denoted $k_F(v)$, is its degree in its tree, so if $v \in v(T_i)$ then $k_F(v)=k_{T_i}(v)$. For $F$, we let $F^{\downarrow}$ be the sequence of reordering $\{T_1, \cdots, T_c\}$ in decreasing order of number of nodes, breaking ties by the original order of appearance in $F$.

For $i \ge 0$ let $s^i(T) = \#\{v \in v(T): k_T(v)=i\}$ and define $s^i(F)$ accordingly, so $s^i(F)=\sum_{j \le c} s^i(T_j)$. The {\em degree sequences} of $T$ and of $F$ are $s(T) = (s^i(T),i \ge 0)$ and $s(F) = (s^i(F),i \ge 0)$, respectively. Any sequence $s=(s^i, i\ge 0)$ of non-negative integers with $\sum\limits_{i\ge 0} s^i<\infty$ and with $\sum\limits_{i\ge 0} is^i<\sum\limits_{i\ge 0} s^i$ is the degree sequence of some tree or forest. More precisely, writing $c(s):= \sum\limits_{i\ge 0}(1-i)s^i>0$, then any forest with degree sequence $s$ consists of exactly $c(s)$ trees.

The goal of this paper is to study the asymptotic structure of large random forests with a given degree sequence, in the ``supercritical'', finite variance regime. In this setting, the forest typically consists of a single, large tree containing all but a vanishing fraction of the nodes. The scaling limit of this tree is $\cT$, the Brownian Continuum Random Tree (CRT) introduced by Aldous in \cite{AldousI, AldousII, AldousIII}. The remaining nodes form another random forest, which may be expected to have its own scaling limit (with an appropriate rescaling, which should be different from that of the large tree). The contributions of this paper confirm that the above picture is correct, and yield a pleasingly straightforward description, which we now provide, for the joint scaling limit of the large tree and the small trees. 

Let $B=(B(t),t \ge 0)$ be a linear Brownian motion. 
For $t \ge 0$ let $R(t) = B(t) - \inf(B(s),s \le t)$; the process $R=(R(t),t \ge 0)$ is Brownian motion reflected at its running minimum. 
Let $Z = \{t \ge 0: R(t) = 0\}$ be the zero set of $R$.  By definition, this is also the set of times at which $B$ is equal to its running minimum.

Now let $\tau(x) = \inf(t: B(t) \le -x)$ for $x \ge 0$, and let $Z(x) = Z \cap [0,\tau(x)]$. For $\sigma>0$, the relative complement $[0,\tau(\frac{1}{\sigma})]\setminus Z(\frac{1}{\sigma})$ is almost surely a countable collection of intervals with distinct lengths, and with total length $\tau(\frac{1}{\sigma})$. List these intervals in decreasing order of length as $((g_i,d_i),i \ge 1)$. 

For $i \ge 1$ let $\cT_i$ be the continuum random tree coded by $2B_i$, where
\[
B_i = (B(g_i+t)-B(g_i),0 \le t \le d_i-g_i) = (R(g_i+t)-R(g_i),0 \le t \le d_i-g_i)\, ;
\]
this construction is described in more detail and generality in Section \ref{sec:concepts}. Then the scaling limit of the small trees has the law of the sequence $\cF = (\cT^\downarrow_i,i \ge 1)$, which is a decreasing reordering of $(\cT_i, i\ge 1)$ according to $(d_i-g_i, i\ge 1)$. 

For any probability distribution $q=(q^{(i)}, i\ge 0)$ on $\mathbb{N}$, we let $\sigma^2(q)=\sum\limits_{i\ge 0}i^2 q^{(i)}$.
\begin{thm}\label{thm:main} 
Fix a sequence $p=(p^i,i \ge 0)$ with $\sum_{i \ge 0} p^i=1=\sum_{i \ge 0} i p^i$ and with $\sigma^2:=\sigma^2(p) \in (0,\infty)$. 
For each $n \ge 1$ let $s_n = (s_n^i,i \ge 0)$ be a degree sequence with $\sum_{i \ge 0} s_n^i = n$ and write $p_n=(p^i_n,i\ge 0)=(s^i_n/n, i\ge 0)$ and $c_n=c(s_n)$. 

Let $F_n$ be a uniformly random plane forest with degree sequence $s_n$. 
Let $\hat{F}_n=(T_{n,i}^{\downarrow}, 2\le i \le c_n)$ be the decreasing reordering of $F_n$, excluding the largest tree $T^\downarrow_{n,1}$. Suppose that $p_n \to p$ in $L^2$ and $c_n=o(n^{1/2})$, then 
\[
\left(\frac{\sigma(p_n)T^\downarrow_{n,1}}{n^{1/2}}, \frac{\sigma(p_n)\hat{F}_n}{c_n}, \frac{n-|T^\downarrow_{n,1}|}{c^2_n}\right)\overset{d}{\rightarrow}\left(\cT, \cF, \tau\left(\frac{1}{\sigma}\right)\right)
\]
where the first coordinate of the joint convergence is in the GHP sense, the second coordinate is in the sense of coordinatewise GHP convergence, and $\cT$ and $\cF$ are independent.
\end{thm}

\noindent {\bf Remarks.} \\
$\bullet$ The condition that $\sum_{i \ge 0} s_n^i = n$ in Theorem~\ref{thm:main} is for notational convenience; all proofs carry through with only cosmetic changes provided that $|s_n| = \sum_{i \ge 0} s_n^i \to \infty$ as $n \to \infty$, that $|s_n|^{-1}\cdot s_n \to p$ in $L^2$ and that $c_n = o(|s_n|^{1/2})$. \\
$\bullet$ Fix a critical, finite variance offspring distribution $\nu$, and let $\mathcal{F}_n$ be a forest of $c_n$ independent Galton-Watson$(\nu)$ trees with offspring distribution $\nu$, conditioned to have total progeny $n$. It is not hard to check, as in \cite{BroutinMarckert2014}, that with high probability the degree sequence of $\mathcal{F}_n$ satisfies the conditions of Theorem~\ref{thm:main}, so the distributional convergence of the theorem also applies to $\mathcal{F}_n$. The convergence of the third coordinate, in the Galton-Watson setting, appears as Theorem~2.1.5 of \cite{Pavlov2000}, and provides a new proof and different perspective on that result; the convergence of the second coordinate strengthens and generalizes and removes a moment assumption from Theorem~1.7 of \cite{Cheplyukova1998}. 

\medskip 

The field of scaling limit of large random structures is motivated by the seminal papers \cite{AldousI, AldousII, AldousIII} by Aldous, where he introduced the concept of Brownian Continuum Random Tree and showed that critical Galton-Watson tree conditioned on its size has CRT as limiting object. To be more specific, our work here is a natural generalization of \cite{BroutinMarckert2014} where it is shown that under natural hypotheses on the degree sequences, after suitable normalization, uniformly random trees with given degree sequence converge to CRT as sizes of trees tend to infinity. Let $n$ be the number of nodes of the forest. In this paper we work on uniformly random forests where the number of trees is $o(n^{1/2})$; a previous paper \cite{Lei2017+} addressed the asymptotic behaviour of such forests in the regime where the number of trees is of order $n^{1/2}$. 
%On the other hand, our work is also motivated by the work of Pavlov \cite{Pavlov2000}, who systematically studied the  {\em simply generated forest} $\mathfrak{F}_{N,n}$, a forest with $N+n$ nodes and consisting of $N$ trees from a simply generated family of trees. Our work gives an alternative probabilistic proof for some of Pavlov's results on the maximum size of a tree in $\mathfrak{F}_{N,n}$.

\subsection*{Outline of the section}
In the remainder of this section, we first briefly introduce the concepts required to understand the statement of Theorem \ref{thm:main} rigorously. In Section~\ref{sec:intro_key_ingredients} we describe the key ingredients of the proof of our main theorem. In Section \ref{sec:proof of main theorem} we explain how  to deduce Theorem \ref{thm:main} from the results of Section~\ref{sec:intro_key_ingredients}, and outline the remaining sections of the article.

\subsection{Concepts}\label{sec:concepts}
\subsection*{Real trees}
We briefly recall the concepts of real trees and real trees coded by continuous functions, which are necessary for understanding the construction of $\cF$. A more lengthy presentation about the probabilistic aspects of real trees can be found in \cite{Evans2008, LeGall2005}.

\begin{dfn}
A compact metric space $(T, d)$ is a {\em real tree} if the following hold for every $a, b\in T$:

(i) There is a unique isometric map $f_{a,b}:[0, d(a,b)]\to T$ such that $f_{a,b}(0)=a$ and $f_{a,b}(d(a,b))=b$.

(ii) If $q$ is a continuous injective map from $[0,1]$ into $T$, such that $q(0)=a$ and $q(1)=b$, we have $q([0,1])=f_{a,b}([0,d(a,b)])$.
\end{dfn}

Now we show a way of constructing real trees from continuous functions. Let $g:[0,\infty)\rightarrow [0,\infty)$ be a continuous function with compact support and such that $g(0)=0$. For every $s, t\ge 0$, let $$d^\circ_g(s,t)=g(s)+g(t)-2m_g(s,t)$$ where $$m_g(s,t)=\min\limits_{s\wedge t\le r\le s\vee t} g(r).$$ The function $d^\circ_g$ is a pseudometric on $[0,\infty)$. Define an equivalence relation $\sim$ on $[0,\infty)$ by setting $s\sim t$ iff $d^\circ_g(s,t)=0$. Then let $T_g=[0,\infty)/\sim$ and let $d_g$ be the induced distance on $T_g$. Then $(T_g, d_g)$ is a real tree (see, e.g. Theorem 2.2 in \cite{LeGall2005}).

%With this construction, let $\cT_i$ be the real tree encoded by the twice of the following continuous function \[B_i = (B(g_i+t)-B(g_i),0 \le t \le d_i-g_i) = (R(g_i+t)-R(g_i),0 \le t \le d_i-g_i).\] Then we let $\cF=(\cT^\downarrow_i, i\ge 1)$ where $(\cT^\downarrow_i,i\ge 1)$ is a decreasing reordering of $(\cT_i, i\ge 1)$ according to $(d_i-g_i,i\ge 1)$, the lengths of excursion intervals.

\subsection*{GHP convergence} 
Let $(X, d)$ and $(X', d')$ be compact metric spaces. Then the {\em Gromov-Hausdorff distance} between $(X, d)$ and $(X', d')$ is given by $$d_{GH}((X,d),(X',d'))=\inf\limits_{\phi,\phi', Z}d_H^Z(\phi(X),\phi'(X')),$$ where the infimum is taken over all isometric embeddings $\phi:X\hookrightarrow Z$ and $\phi': X'\hookrightarrow Z$ into some common Polish metric space $(Z, d^Z)$ and $d_H^Z$ denotes the Hausdorff distance between compact subsets of $Z$, that is, $$d_H^Z(A,B)=\inf\{\epsilon>0: A\subset B^\epsilon, B\subset A^\epsilon\},$$ where $A^\epsilon$ is the {\em $\epsilon-$enlargement} of $A$: $$A^\epsilon=\{z\in Z: \inf\limits_{y\in A} d^Z(y, z)<\epsilon\}.$$

Note that strictly speaking $d_{GH}$ is not a distance since different compact metric spaces can have GH distance zero.

A {\em measured metric space} $\mathcal{X}=(X, d, \mu)$ is a metric space $(X, d)$ with a finite Borel measure $\mu$. Let $\mathcal{X}=(X, d, \mu)$ and 
${\mathcal{X}}'=(X', d',\mu')$ be two compact measured metric spaces, they are {\em GHP-isometric} if there exists an isometric one-to-one map $\Phi: X\rightarrow X'$ such that $\Phi_{\ast}\mu=\mu'$ where $\Phi_{\ast}\mu$ is the {\em push forward} of measure $\mu$ to $(X', d')$, that is, $\Phi_{\ast}\mu(A)=\mu(\Phi^{-1}(A))$ for $A\in\mathcal{B}(X')$. In this case, call $\Phi$ a {\em GHP-isometry}. Suppose both $\mathcal{X}$ and $\mathcal{X'}$ are compact, then define the Gromov-Hausdorff-Prokhorov distance as:
$$d_{GHP}(\mathcal{X},\mathcal{X'})=\inf\limits_{\Phi, \Phi', Z}(d_H^Z(\Phi(X), \Phi'(X'))+d_P^Z(\Phi_\ast\mu, \Phi'_\ast\mu'))$$ where the infimum is taken over all GHP-isometric embeddings $\Phi: X\hookrightarrow Z$ and $\Phi': X'\hookrightarrow Z$ into some common Polish metric space $(Z, d^Z)$, and $d_P^Z$ denotes the Prokhorov distance between finite Borel measures on $Z$, that is, \[d_P^Z(\mu, \nu)=\inf\{\epsilon>0: \mu(A)\le\nu(A^\epsilon)+\epsilon, \nu(A)\le \mu(A^\epsilon)+\epsilon \mbox{ for any closed set }A\}.\]
Let $\mathbb{K}$ denote the set of GHP-isometry classes of compact measured metric spaces and we identify $\mathcal{X}$ with its GHP-isometry class. 
\begin{thm}[Theorem 2.5 in \cite{AbrahamDH2013}]
The function $d_{GHP}$ defines a metric on $\mathbb{K}$ and the space $(\mathbb{K}, d_{GHP})$ is a Polish metric space.
\end{thm}
%Then for $\mathcal{X}_n\in\mathbb{K},\mathcal{X}\in\mathbb{K}$, we say $\mathcal{X}_n$ converges to $\mathcal{X}$ {\em in GHP sense} if \[d_{GHP}(\mathcal{X}_n, \mathcal{X})\to 0, \mbox{ as } n\to\infty.\]
We next define coordinatewise GHP convergence of sequences of measured metric spaces. For $\mathbf{X}_n=(\mathcal{X}_{n,j}, j\ge 1), \mathbf{X}=(\mathcal{X}_j, j\ge 1)$ in $\mathbb{K}^{\mathbb{N}}$, we say that $\mathbf{X}_n$ converges to $\mathbf{X}$ in {\em coordinatewise GHP sense} if for any $j\in\mathbb{N}$, \[\sup\limits_{1\le l\le j} d_{GHP}(\mathcal{X}_{n,l}, \mathcal{X}_l)\to 0 \mbox{ as } n\to\infty.\]

Now to understand the statement of Theorem \ref{thm:main} in the rigorous way, we are viewing $T^\downarrow_{n,1}$ and each tree components of $\hat{F}_n$ as measured metric space where the distance is rescaled graph distance and the measure is the uniform measure putting mass $1/n$ on each node of $T^\downarrow_{n,1}$ and $\hat{F}_n$.

\subsection{Functional convergence and proof of Theorem \ref{thm:main}}\label{sec:intro_key_ingredients}

Given a degree sequence $s=(s^i, i\ge 0)$ with $|s|=n$, we let $d(s)\in\mathbb{Z}^n_{\ge 0}$ be the vector whose entries are weakly increasing and with $s^i$ entries equal to $i$, for each $i\ge 0$. For example, if $s=(3, 2, 0, 1, 0,0, \cdots)$ with $s^i=0$ for $i\ge 4$, then $d(s)=(0, 0, 0, 1, 1, 3)$. Suppose we have a sequence of degree sequences $(s_n)_{n\in\mathbb{N}}$, with $s_n=(s^i_n, i\ge 0), |s_n|=n, c_n:=c(s_n)=o(n^{1/2})$ and $n^{-1}\cdot s_n\to p$ in $L^2$ for some distribution $p=(p^i, i\ge 0)$ with mean 1 and finite variance $\sigma^2$ on $\mathbb{N}$. Let $C_{n,1}, \cdots, C_{n,n}$ be a uniformly random permutation of $d(s_n)$. For $1\le k\le n$, let $X_{n,k}=C_{n,k}-1$, and set $S_{n,k}=\sum_{j=1}^k X_{n,j}$. Our proof makes use of the following functional convergence theorem. 
\begin{thm}\label{thm:walk convergence}
We have the following convergence of processes: 
\begin{equation}\label{eqn:walk convergence}
\left(\frac{1}{c_n}S_{n, \lfloor tc^2_n \rfloor}, t\ge 0 \right)\overset{d}{\to}\left(\sigma B(t), t\ge 0\right)
\end{equation}
where $(B(t), t\ge 0)$ is standard Brownian Motion.
\end{thm}
Theorem~\ref{thm:walk convergence} will yield a description of the asymptotic behaviour of the sizes of all but the largest tree of $F_n$. 
\begin{cor}\label{cor: small tree sizes}
We have \[\left(\frac{|T^\downarrow_{n,i+1}|}{c^2_n}, i\ge 1\right)\overset{d}{\to}\left(g_i-d_i, i\ge 1\right)\mbox{ in } L^1\] where $((g_i,d_i), i\ge 1)$ are the excursion intervals of $(R(t), t\le\tau(\frac{1}{\sigma}))$ in decreasing order of length.
\end{cor}
Corollary \ref{cor: small tree sizes} is equivalent to the assertions that
\begin{equation}\label{eqn:sum of small tree sizes}
\frac{1}{c^2_n}\sum\limits_{i\ge2}|T^\downarrow_{n,i}|\overset{d}{\to}\tau(\frac{1}{\sigma})=\sum\limits_{i\ge 1}(g_i-d_i)
\end{equation}
and that for any fixed $j\in\mathbb{N}$,
\begin{equation}\label{eqn:tree size convergence}
\left(\frac{|T^\downarrow_{n,2}|}{c^2_n},\frac{|T^\downarrow_{n,3}|}{c^2_n},\ldots, \frac{|T^\downarrow_{n,j}|}{c^2_n}\right)\overset{d}{\to} (g_1-d_1, g_2-d_2,\ldots, g_{j-1}-d_{j-1}).
\end{equation}

We will prove this corollary in Section \ref{sec:convergence of the processes}.
To describe the limit structure of each tree, we appeal to the following theorem in \cite{BroutinMarckert2014}.

\begin{thm}\label{thm:BroutinMarckert}
Let $\{s_n, n\ge 1\}$ be a degree sequence such that $|s_n|=n\rightarrow\infty, \Delta_n:=\max\{i: s^i_n\neq 0\}=o(n^{1/2})$. Suppose that there exists a distribution $p=(p^i, i\ge 0)$ on $\mathbb{N}$ with mean 1 such that $p_n=(s^i_n/n, i\ge 0)$ converges to $p$ coordinatewise and such that $\sigma(p_n)\rightarrow\sigma(p)\in(0,\infty)$. Let $\mathbb{T}_n$ be the random plane tree under $\mathbb{P}_{s_n}$, the uniform measure on the set of plane trees with degree sequence $s_n$ and let $d_{\mathbb{T}_n}$ be the graph distance in $\mathbb{T}_n$. Then when $n\to\infty$, \[(\mathbb{T}_n, \frac{\sigma(p_n)}{\sqrt{n}}d_{\mathbb{T}_n})\overset{d}{\to}\mathcal{T}\] in the GHP sense. 
\end{thm}

To apply Theorem \ref{thm:BroutinMarckert} to each $T^\downarrow_{n,i}$, we also need to verify that the assumptions of Theorem \ref{thm:main} hold. 
For fixed integers $i\ge 0$ and $l \ge 1$, let \[p^i_{n,l}:=\frac{|\{v\in T^\downarrow_{n,l}:k(v)=i\}|}{|T^\downarrow_{n,l}|} \mbox{ and } p_{n,l}=(p^i_{n,l}, i\ge 0).\]

In Section \ref{sec:verify degree conditions} we prove the following assertions:
\begin{equation}\label{eqn:degree proportion converge} 
\mbox{for any fixed } i\ge 0\mbox{ and }l\ge 1,\ p^i_{n,l}-p^i_n\overset{p}{\to}0, \mbox{ as } n\to\infty,
\end{equation}
and
\begin{equation}\label{eqn: sigma square converge}
\mbox{for any } l\ge 1, \ \sigma^2(p_{n,l})-\sigma^2(p_n)\overset{p}{\to} 0, \mbox{ as } n\to\infty.
\end{equation}

Note that once these two conditions are verified, it follows that for any fixed $l\ge 1$, \[\max\{i: p^i_{n,l}\neq 0\}=o_p(|T^\downarrow_{n,l}|^{1/2})\mbox{ as } n\to\infty;\] see, e.g. Lemma A.1 in \cite{Lei2017+}.

\subsection{Proof of Theorem \ref{thm:main}}\label{sec:proof of main theorem}
Now we are ready to give the proof of Theorem \ref{thm:main}, assuming the results of Section \ref{sec:intro_key_ingredients}.
\begin{proof}
It suffices to prove that for any fixed $j\in\mathbb{N}$, 
\[\left(\frac{\sigma(p_n) T^\downarrow_{n,1}}{n^{1/2}}, \left(\frac{\sigma(p_n) T^\downarrow_{n,l}}{c_n}, 2\le l\le j\right), \frac{n-|T^\downarrow_{n,1}|}{c^2_n}\right)\overset{d}{\to}\left(\cT, (\cT^\downarrow_1,\cdots,\cT^\downarrow_{j-1}), \tau\left(\frac{1}{\sigma}\right)\right).\]
The convergence of the third coordinate is simply (\ref{eqn:sum of small tree sizes}). This in particular implies that $\frac{|T^\downarrow_{n,1}|}{n}\overset{p}{\to} 1$. Since $p_n\to p$ in $L^2$, it straightforwardly follows that with probability $1-o(1)$, $T^\downarrow_{n,1}$ satisfies the conditions of Theorem \ref{thm:BroutinMarckert}; this yields the convergence of the first coordinate. With (\ref{eqn:degree proportion converge}) and (\ref{eqn: sigma square converge}), we can also apply Theorem \ref{thm:BroutinMarckert} to each $T^\downarrow_{n,l}$ with $l\ge 2$ and conclude that
\[
\frac{\sigma(p_{n,l})}{|T^\downarrow_{n,l}|^{1/2}}T^\downarrow_{n,l}\overset{d}{\to}\cT.
\]
Since the trees $(T^\downarrow_{n,l}, l\ge 1)$ are conditionally independent given their degree sequences, it follows that 
\[\left(\frac{\sigma(p_{n,l})}{|T^\downarrow_{n,l}|^{1/2}}T^\downarrow_{n,l}, 2\le l\le j\right)\overset{d}{\to}\left(\tilde{\cT}_{l-1}, 2\le l\le j\right),\]
where $(\tilde{\cT}_l)_{l\in\mathbb{N}}$ are independent copies of $\cT$.
Using (\ref{eqn: sigma square converge}) again, together with (\ref{eqn:tree size convergence}) and Brownian scaling, the convergence of the second coordinate then follows.
\end{proof}

\subsection*{Outline of the rest of the paper}
In Section \ref{sec:combinatorics} we describe a combinatorial construction which associates a {\em marked cyclic forest} with the concatenation of a sequence of first passage bridges, followed by one lattice bridge. This construction is what links Theorem \ref{thm:walk convergence} with random forests. In Section \ref{sec:convergence of the processes} we give the proof of Theorem \ref{thm:walk convergence} and Corollary \ref{cor: small tree sizes}. Finally in Section \ref{sec:verify degree conditions} we prove (\ref{eqn:degree proportion converge}) and (\ref{eqn: sigma square converge}) using martingale concentration inequalities. 

\section{\bf Coding marked forests by lattice paths}\label{sec:combinatorics}

We call a sequence of integers $\mathbf{b}=(b(0), b(1), \ldots, b(n))$ a {\em lattice bridge} if \[b(0)=0, b(n)=-1 \mbox{ and } \forall 0\le i\le n-1, ~b(i+1)-b(i)\ge -1.\] If $\mathbf{b}$ is a lattice bridge and $\min\limits_i\{i:b(i)=-1\}=n$, then we call $\mathbf{b}$ a {\em first passage bridge}. 
Given a lattice bridge $\mathbf{b}$ and a positive integer $k\le n$, we define a lattice bridge $\mathbf{b}^{(k)}$ as follows. First, for $1\le i\le n$, let $b(n+i)=b(n)+b(i)=-1+b(i)$. Then for $0\le i\le n$, let $\mathbf{b}^{(k)}(i)=b(k+i)-b(k)$.
%Let $\mathbf{x}=(x_1,\ldots, x_n)$ where $x_i=b_i-b_{i-1}, 1\le i\le n$, we also call $\mathbf{b}$ the {\em walk with step $\mathbf{x}$}. Let $\mathbf{x}^{(i)}$ denote the {\em cyclic shift} of $\mathbf{x}$, that is, $(x_{j+i \,(\mathrm{mod}\,n)}, 1\le i\le n)$. 
Let $[n]=\{1,\cdots,n\}$.
We have the following elementary lemma as a variant of the classical ballot theorem.

\begin{lem}[Lemma 6.1 in \cite{Pitman2006}]\label{lem:rotation}
Fix a lattice bridge $\mathbf{b}=(b(i), 0\le i\le n)$, and let $r=r(\mathbf{b})\in [n]$ be minimal so that $b(r)=\min (b(i), i\le n)$. Then $\mathbf{b}^{(r)}$ is a first passage bridge, and $r$ is the only such value in $[n]$.
\end{lem}

Lemma \ref{lem:rotation} is illustrated by Figure \ref{fig: latticebridge_firstpassagebridge_marketree}(a) and Figure \ref{fig: latticebridge_firstpassagebridge_marketree}(b). In Figure \ref{fig: latticebridge_firstpassagebridge_marketree}(a) we have a lattice bridge $\mathbf{b}=(0,-1,-1,-2,-1,1,0,-1)$. The vertical dashed line shows the position of $\mathbf{b}$ attaining its minimum for the first time, hence the unique position for the cyclic shift to transform $\mathbf{b}$ to a first passage bridge, as claimed by Lemma \ref{lem:rotation}. The resulting first lattice bridge, with steps $\mathbf{b}^{(3)}$, is shown in Figure \ref{fig: latticebridge_firstpassagebridge_marketree}(b).

%\begin{figure}[h]
%\centering
%\begin{minipage}{0.5\textwidth}
%\centering
%\includegraphics[scale=0.6]{lattice_bridge_new.png}
%\caption{Lattice bridge $\mathbf{b}$}
%\label{fig: lattice bridge}
%\end{minipage}\hfill
%\begin{minipage}{0.5\textwidth}
%\centering
%\includegraphics[scale=0.6]{marked_tree_new.png}
%\caption{Marked tree $(T,v)$}
%\label{fig: marked tree}
%\end{minipage}
%
%
%\begin{minipage}{0.5\textwidth}
%\centering
%\includegraphics[scale=0.6]{first_passage_bridge_new.png}
%\caption{Fig.(b): First passage bridge $\mathbf{b}'$}
%\label{fig: first passage bridge}
%\end{minipage}\hfill
%\end{figure}  

\begin{figure}[h]
\centering
\begin{minipage}{\textwidth}
\includegraphics[scale=1]{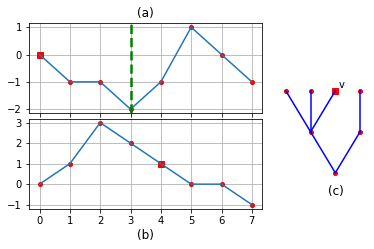}
\caption{(a) a lattice bridge; (b) the corresponding marked first-passage bridge; (c) the corresponding marked tree.}
\label{fig: latticebridge_firstpassagebridge_marketree}
\end{minipage}
\end{figure}

A plane tree is a rooted tree $T$ in which the children of each node have a left-to-right order. Recall that for a plane tree $T$ and a node $v\in v(T)$, we write $k_T(v)$ to denote the degree of $v$ in $T$. We also write $\mathrm{lex}(T)=(k_T(u_1), \ldots, k_T(u_{|T|}))$ where $(u_i, 1\le i\le |T|)=(u_i(T), 1\le i\le |T|)$ are nodes of $T$ listed in lexicographic order. 

%It is well known that there is a bijection between first passage bridges and plane trees. For example, the first passage bridge $\mathbf{b}'=(0,1,3,2,1,0,0,-1)$ in Figure \ref{fig: first passage bridge} corresponds to the tree $T$ in Figure \ref{fig: tree} where $\mathbf{b}'$ is the {\em depth-first walk} of $T$. Formally
For any sequence $\mathbf{c}=(c_1, \cdots, c_n)\in\mathbb{R}^n$, we write $W_{\mathbf{c}}(j)=\sum\limits_{i=1}^j(c_i-1)$ for $j\in [n]$, let $W_{\mathbf{c}}(0)=0$ and make $W_{\mathbf{c}}$ a continuous function on $[0, n]$ by linear interpolation. 
%Then Figure \ref{fig: first passage bridge} depicts the depth-first walk $W_{\mathrm{lex}(T)}$ of $T$.
A classical bijection between plane trees and first passage bridges associates to a tree $T$ its {\em depth-first walk} $(W_{\mathrm{lex}(T)}(i), 0\le i\le n)$, see, e.g. Chapter 6 of \cite{Pitman2006}. We build on this bijection below.

For a plane tree $T$ and $v\in v(T)$, we call the pair $(T, v)$ a {\em marked tree} and call $v$ the {\em mark}. The bijection between first passage bridges and plane trees also leads to a bijection between lattice bridges and marked trees. This bijection, depicted in Figure \ref{fig: latticebridge_firstpassagebridge_marketree}, is specified as follows. For a lattice bridge $\mathbf{b}$, let $r=r(\mathbf{b})$ as in Lemma \ref{lem:rotation}, let $\mathbf{b}'=\mathbf{b}^{(r)}$ be the first passage bridge corresponding to $\mathbf{b}$, and let $T$ be the plane tree with depth-first walk $\mathbf{b}'$. Then the marked node is $v=u_{|T|-r+1}(T)$, the $(|T|-r+1)$'st node of $T$ in lexicographic order. The mark $v$ is denoted by a red square in Figure \ref{fig: latticebridge_firstpassagebridge_marketree}. 

%To see this, for $\mathbf{b}$ in Figure \ref{fig: lattice bridge}, we let tree $T$ be the tree corresponding to the unique first passage bridge $\mathbf{b}'$ (i.e., the unique $\mathbf{b}^{(k)}$ as in Lemma \ref{lem:rotation}). And let $v$ be the node of $T$ corresponding to the step in $\mathbf{b}'$, which is the image of the first step of $\mathbf{b}$ under the cyclic shift in Lemma \ref{lem:rotation}. These steps are shown as yellow dashed edges in Figure \ref{fig: lattice bridge} and Figure \ref{fig: first passage bridge}. And the marked tree corresponding to $\mathbf{b}$ is hence $(T, v)$, as in Figure \ref{fig: marked tree}.

A {\em marked forest} is a pair $(F, v)$ where $F$ is a plane forest and $v\in v(F)$. We refer $v$ as the {\em mark} of $(F,v)$. A {\em marked cyclic forest} is a marked forest with its mark in its last tree; the name is because we can equivalently view such a forest as having its trees arranged around a cycle.

Fix an integer sequence $W=(W_i:0\le i\le n)$ with $W_0=0, W_n=-k$, and $W_i-W_{i-1}\ge -1$ for all $1\le i\le n$. 
The bijections described above allow us to view $W$ as a marked cyclic forest $(F,v)=(F(W),v(W))$ consisting of $k-1$ trees and one marked tree, as follows. For integer $b<0$, let $\tau(b)=\inf\{t\in\mathbb{N}: W_t\le b\}$. For $1\le j\le k-1$, let $T_j$ be the tree whose depth-first walk is $(W_i-W_{\tau(-(j-1))}: \tau(-(j-1))\le i\le \tau(-j))$.  Let $(T_k,v)$ be the marked tree corresponding to lattice bridge $(W_i-W_{\tau(-(k-1))}: \tau(-(k-1))\le i\le n)$.  Then $(F(W),v(W))=((T_1,\ldots,T_k),v)$.  We call $W$ the {\em coding walk} of the forest, and note that the coding is bijective: $W$ can be recovered from $(F(W),v(W))$ as the concatenation of the first-passage bridges which code $T_1,\ldots,T_{k-1}$ and the lattice bridge which codes $(T_k,v)$. This bijection is illustrated in Figure \ref{fig:bridge to minus k} and Figure \ref{fig:marked forest}. In Figure \ref{fig:bridge to minus k} the whole sequence is decomposed into three segments (divided by vertical dashed lines). The first two segments are first passage bridges, hence correspond to plane trees $T_1, T_2$. The last part is a lattice bridge, hence corresponds to a marked tree $(T_3, v)$ and the node $v$ is again depicted by a square mark. These trees are shown in Figure \ref{fig:marked forest}. 

\begin{figure}[h]
\centering
\begin{minipage}{\textwidth}
\centering
\includegraphics[scale=0.9]{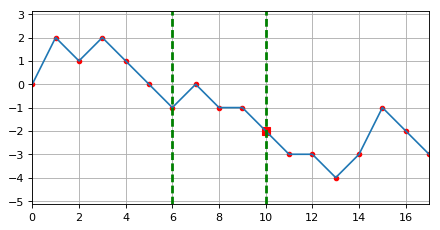}
\caption{A lattice walk $W=(W_i: 0\le i\le 17)$.}
\label{fig:bridge to minus k}
\end{minipage}\vfill
\begin{minipage}{\textwidth}
\centering
\includegraphics[scale=1]{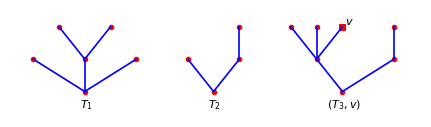}
\caption{The marked forest $(F(W),v(W))=((T_1, T_2, T_3), v)$.}
\label{fig:marked forest}
\end{minipage}
\end{figure}

Given a degree sequence $s=(s^{(i)}, i\ge 0)$ with $\sum_{i\ge 0}s^i=n$, recall from Section \ref{sec:intro_key_ingredients} that $d(s)\in\mathbb{Z}_{\ge 0}^n$ is the vector whose entries are weakly increasing and with $s^{(i)}$ entries equal to $i$, for each $i\ge 0$. 
Let $\mathrm{D}(s)$ be the set of sequences $d\in\mathbb{Z}_{\ge 0}^n$ which are permutations of $d(s)$ (there are $n!/(\prod_{i}s^i!)$ of them).
Let $\mathrm{MCF}(s)$ be the set of all marked cyclic forests with degree sequence $s$. By the correspondence we developed previously, the following is lemma is immediate. 

\begin{lem}\label{lem:bijection}
Fix a degree sequence $s=(s^{(i)}, i\ge 0)$ with $\sum_{i\ge 0}s^i=n$. Let $\pi$ be a uniformly random permutation of $[n]$, and let $W=W_{\pi(d(s))}$. Then the marked cyclic forest $(F(W),v(W))$ coded by $W$ is uniformly distributed on $\mathrm{MCF}(s)$. 
\end{lem}
In particular, we have the following corollary.
\begin{cor}
Let $s=(s^{(i)}, i\ge 0)$ with $\sum_{i\ge 0}s^i=n$. Let $(F,v)$ be a uniformly random element of $\mathrm{MCF}(s)$, and let $M$ be the total number of nodes in the non-marked trees of $(F,v)$. Let $\pi$ be a uniformly random permutation of $[n]$ and let $S:[0, n]\rightarrow \mathbb{R}, ~S(t)=W_{\pi(d(s))}(t)$. Then 
\begin{equation}\label{eqn:size of small trees}
M\overset{d}{=}\inf\{t: S(t)=-c(s)+1\}.
\end{equation}
\end{cor}
 We will also need the following easy fact connecting linear forests with marked cyclic forests.
\begin{lem}\label{lem:maps forest to marked cyclic forest}
Fix a degree sequence $s=(s^{(i)}, i\ge 0)$, and let $F$ be a uniformly random linear forest with degree sequence $s$, and let $(F^*,v)$ be the marked cyclic forest obtained from $F$ by marking a uniformly random node and applying the requisite cyclic shift of the trees of $F$. Then $(F^*,v)$ is a uniformly random element of $\mathrm{MCF}(s)$. 
\end{lem}
\begin{proof}
Let $\mathrm{F}(s)$ be the set of all plane forests with degree sequence $s$. 
The operation of marking a node induces an $n$-to-$c(s)$ map from $\mathrm{F}(s)$ to $\mathrm{MCF}(s)$, from which the lemma is immediate.
\end{proof}

The preceding lemma allows us to relate the random forest $F_n$ from Theorem \ref{thm:main} with the lattice path $S_n=(S_{n,k},0 \le k \le n)$ from Theorem \ref{thm:walk convergence}. Let $(F_n^*,v_n)=((T_{n,k},1 \le k \le c_n),v_n)$ be obtained from $F_n$ by marking a uniformly random node and applying the requisite cyclic shift of the trees of $F_n$. Then we may couple $F_n$ and $S_n$ so that $S_n = (S_{n,j},0 \le j \le n)$ is the coding walk of $(F_n^*,v_n)$. We work with such a coupling for the remainder of the paper.

\section{\bf Convergence of the coding processes}\label{sec:convergence of the processes}
The goal of this section is to prove Theorem \ref{thm:walk convergence} and Corollary \ref{cor: small tree sizes}. To achieve that, we decompose the walk process into two random processes. To be precise, let $d_n:=\frac{n^{1/2}}{c_n}$ and fix a sequence $(t_n)_{n\in\mathbb{N}}$, such that $t_n=o(d_n)$ and $t_n=\omega(1)$. This is possible since $d_n\to\infty$ as $n\to\infty$ by our assumption that $c_n=o(n^{1/2})$. We consider the following two processes. Let $(M_{n,k}, k\le n)$ be as follows, $M_{n,0}=0$ and for $k\ge 1$, \[M_{n,k}-M_{n, k-1}=X_{n,k}\mathbbm{1}_{|X_{n,k}|<t_n}.\] Similarly, let $(R_{n,k}, k\le n)$ be given by $R_{n,0}=0$, and for $k\ge 1$, \[R_{n,k}-R_{n, k-1}=X_{n,k}\mathbbm{1}_{|X_{n,k}|\ge t_n}.\] Then clearly we have $S_{n,k}=M_{n,k}+R_{n,k}$ for all $k\le n$. Define the following quantity: \[\mu^+_n:=\sum\limits_{j\ge t_n+1}(j-1) \frac{s^j_n}{n}.\]

Theorem \ref{thm:walk convergence} is an immediate consequence of the following two results: 
\begin{equation}\label{eqn:M process convergence}
\left(\frac{1}{c_n}(M_{n, \lfloor tc^2_n \rfloor}+\mu^+_n\lfloor tc^2_n \rfloor), t\ge 0 \right)\overset{d}{\to}\left(\sigma B(t), t\ge 0\right)
\end{equation}
and
\begin{equation}\label{eqn:R process convergence}
\left(\frac{1}{c_n}(R_{n, \lfloor tc^2_n \rfloor}-\mu^+_n\lfloor tc^2_n \rfloor), t\ge 0 \right)\overset{d}{\to} 0,
\end{equation}
where $0$ denotes a process $Z$ such that $\p{Z(t)=0,~ \forall t\ge 0}=1$.
For (\ref{eqn:M process convergence}), we are going to use the following theorem from \cite{DiFr1980}.
\begin{thm}[Theorem 4 in \cite{DiFr1980}]\label{thm:finite exchangeable sequences}
Suppose an urn $U$ contains $n$ balls, each marked by one or another element of the set $S$, whose cardinality $c$ is finite. Let $H_{Uk}$ be the distribution of $k$ draws made at random without replacement from $U$, and $M_{Uk}$ be the distribution of $k$ draws made at random with replacement. Then the two probabilities on $S^k$ satisfy
\[||H_{Uk}-M_{Uk}||\le 2ck/n,\]
where $||\cdot||$ denotes the total variation distance.
\end{thm}
\begin{proof}[Proof of (\ref{eqn:M process convergence})]
Let $(\tilde{X}_{n,k}, k\le n)$ be i.i.d. with the law of $X_{n,1}\mathbbm{1}_{|X_{n,1}|<t_n}$, set $\tilde{M}_{n,0}=0$ and for $k\ge 1$, let \[\tilde{M}_{n,k}=\sum\limits_{j=1}^k \tilde{X}_{n,j}.\]
 
Now apply Theorem \ref{thm:finite exchangeable sequences} with urn $U$ containing $n$ balls, with $s^j_n$ balls marked by $j-1$ for $0\le j\le t_n, j\neq 1$, and $s^1_n+\sum\limits_{j>t_n}s^j_n$ balls marked by 0, with $S=\{-1,0,1,\ldots, t_n-1\}$, and with $k=k(n)=n/d_n$. This yields that 
\[||(X_{n,j}\mathbbm{1}_{|X_{n,j}|<t_n}, j\le k)-(\tilde{X}_{n,j}, j\le k)||\le \frac{2t_n(n/d_n)}{n}=2\frac{t_n}{d_n},\]
so for all Borel $B\subset\mathbb{R}^k$,
\[\left|\p{(M_{n,j}, j\le k(n))\in B}-\p{(\tilde{M}_{n,j}, j\le k(n))\in B}\right|\le \frac{2t_n}{d_n}.\]
Since $t_n=o(d_n)$ and $k(n)=n/d_n=d_n\cdot c^2_n=\omega(c^2_n)$, this implies that to establish (\ref{eqn:M process convergence}) it suffices to prove that
\begin{equation}\label{eqn: M iid process convergence}
\left(\frac{1}{c_n}(\tilde{M}_{n, \lfloor tc^2_n \rfloor}+\mu^+_n\lfloor tc^2_n \rfloor), t\ge 0 \right)\overset{d}{\to}\left(\sigma B(t), t\ge 0\right).
\end{equation}
Note that
\[\e{\tilde{X}_{n,1}}=\sum\limits_{j\le t_n}(j-1)\frac{s^j_n}{n}= \frac{1}{n}\sum\limits_{j}(j-1)s^j_n-\sum\limits_{j\ge t_n+1}(j-1)\frac{s^j_n}{n}=-\frac{c_n}{n}-\mu^+_n.\]

Define $\sigma^-_n$ by setting 
\begin{eqnarray}\label{eqn: dfn of sigma-}
(\sigma^-_n)^2:=Var(\tilde{X}_{n,1})&=&\E{\tilde{X}^2_{n,1}}-\E{\tilde{X}_{n,1}}^2\nonumber\\
&=&\sum\limits_{j\le t_n}(j-1)^2\frac{s^j_n}{n}-(-\mu^+_n-\frac{c_n}{n})^2.
\end{eqnarray}

Applying Donsker's theorem to the process $\left(\tilde{M}_{n,k}+k(\mu^+_n+\frac{c_n}{n}), k\ge 0\right)$, we obtain that
\begin{equation}\label{eqn:M process finite convergence}
\left(\frac{1}{a}(\tilde{M}_{n, \lfloor ta^2 \rfloor}+\mu^+_n\lfloor ta^2 \rfloor)+\frac{c_n\lfloor ta^2\rfloor}{na}, t\ge 0\right)\overset{d}{\to}\left(\sigma^-_n B(t), t\ge 0\right),
\end{equation}
as $a\to\infty$.

By our assumption that $n^{-1}\cdot s_n\to p$ in $L^2$ in Theorem \ref{thm:main}, we have $\sum\limits_{j}(j-1)\frac{s^j_n}{n}\to 0$ as $n\to \infty$. Hence for any prescribed $\delta>0$, we can find $L$ large such that $\sum\limits_{j>L}(j-1)\frac{s^j_n}{n}<\delta$. Since $t_n\to\infty$, we must have $\mu^+_n\le\sum\limits_{j>L}(j-1)\frac{s^j_n}{n}<\delta$ for $n$ large enough, i.e. 
\begin{equation}\label{eqn: mean of large degree vanish}
\mu^+_n\to 0 \mbox{ as } n\to\infty. 
\end{equation}
Similarly the assumption that $n^{-1}\cdot s_n\to p$ in $L^2$ implies that \[\sigma^2_n:=\sum\limits_j j(j-1)\frac{s^j_n}{n}\to\sigma^2<\infty,\] so
\begin{equation}\label{eqn:variance of large degree vanish}
(\sigma^+_n)^2\to 0 \mbox{ as } n\to\infty,
\end{equation}
where we let $(\sigma^+_n)^2:=\sum\limits_{j\ge t_n+1}j(j-1)\frac{s^j_n}{n}$.
Using (\ref{eqn: dfn of sigma-}),(\ref{eqn: mean of large degree vanish}) and (\ref{eqn:variance of large degree vanish}), we have 
\begin{equation}\label{eqn:variance convergence}
\sigma_n^2-(\sigma^-_n)^2=(\sigma^+_n)^2-(\mu^+_n+\frac{c_n}{n})(1-\mu^+_n-\frac{c_n}{n})\to 0 \mbox{ as } n\to\infty,
\end{equation}
so $\sigma^-_n\to\sigma$ as $n\to\infty$.
Taking $a=c_n$ in (\ref{eqn:M process finite convergence}), then letting $n\to\infty$, now yields that
\[\left(\frac{1}{c_n}(\tilde{M}_{n, \lfloor tc^2_n \rfloor}+\mu^+_n\lfloor tc^2_n \rfloor)+\frac{\lfloor tc^2_n \rfloor}{n}, t\ge 0 \right)\overset{d}{\to}\left(\sigma B(t), t\ge 0\right).\]
Since $c^2_n=o(n)$, (\ref{eqn: M iid process convergence}) follows.
\end{proof}

To prove (\ref{eqn:R process convergence}), we need the following result concerning \emph{dilation}. Recall (or see, e.g., \cite{Aldous1985}) that given real random variables $U, V$, we say $U$ is a \emph{dilation} of $V$ if there exist random variables $\hat{U}, \hat{V}$ such that \[\hat{U}\overset{d}{=}U,~ \hat{V}\overset{d}{=}V \mbox{ and } \E{\hat{U}|\hat{V}}=\hat{V}.\]

\begin{prop}[Proposition 20.6 in \cite{Aldous1985}]\label{prop:dilution}
Suppose $X_1, \cdots, X_k$ and $X^\ast_1, \cdots, X^\ast_k$ are samples from the same finite population $x_1,\cdots, x_n$, without replacement and with replacement, respectively. Let $S_k=\sum\limits_{i=1}^k X_i, S^\ast_k=\sum\limits_{i=1}^k X^\ast_i$. Then $S^\ast_k$ is a dilation of $S_k$. In particular, $\E{\phi(S^\ast_k)}\ge\E{\phi(S_k)}$ for all continuous convex function $\phi:\mathbb{R}\rightarrow\mathbb{R}$.
\end{prop}

\begin{proof}[Proof of (\ref{eqn:R process convergence})]
We prove that for all $\epsilon>0$, we have
\[\limsup\limits_{n\to\infty}\p{\max\limits_{i\le c^2_n/\epsilon}\left|\frac{R_{n,i}-i\mu^+_n}{c_n}\right|>\epsilon}\le \epsilon,\]
this immediately implies (\ref{eqn:R process convergence}).
Fix $n$ and let $\mathrm{c}_1, \ldots, \mathrm{c}_n$ be such that $|\{1\le k\le n: \mathrm{c}_k=j\}|=s^j_n$. Let $C_1, \ldots, C_n$ be a uniformly random permutation of $\mathrm{c}_1, \ldots, \mathrm{c}_n$. Fix $t_n\in\mathbb{N}$. Define $(R_i, 0\le i\le n)$ as follows: let $R_0=0$, and for $i\ge 0$, let
\[R_{i+1}=\left\{
                  \begin{array}{ll}
                    R_i+C_i-1, & \mbox{ if }C_i\ge t_n+1; \\
                    R_i, & \mbox{ if }C_i\le t_n.
                  \end{array}
                \right.\]
%For $0\le k\le n-1$, let $n^0_k=s^k_n$ and, for $i>0$ let
%\[n^i_k=\left\{
%                  \begin{array}{ll}
%                    n^{i-1}_k, & \mbox{ if }C_i\neq k; \\
%                    n^{i-1}_k-1, & \mbox{ if }C_i=k.
%                  \end{array}
%                \right.\]
For $0\le i\le n$, let $\mathcal{F}_i=\sigma(C_1, \ldots, C_i)$. Since $R_n=n\mu^+_n$ and the process $(R_i, 0\le i\le n)$ has exchangeable increment,
\begin{equation}\label{eqn: cond_exp_of_R_result}
\E{R_{i+1}~|~\mathcal{F}_i}=R_i+\frac{n\mu^+_n-R_i}{n-i}.
\end{equation}
Now let $K_i=\frac{n\mu^+_n-R_i}{n-i}$. Then using (\ref{eqn: cond_exp_of_R_result}), we have \[\E{K_{i+1}~|~\mathcal{F}_i}=\frac{n\mu^+_n-R_i}{n-(i+1)}-\frac{n\mu^+_n-R_i}{(n-i)(n-i+1)}=K_i.\] Hence $K_i$ is an $\mathcal{F}_i-$martingale. 

Since for any $0\le i\le s$, \[\frac{n\mu^+_n-R_i}{n-i}=\mu^+_n+\frac{i\mu^+_n-R_i}{n-i},\] and $\mu^+_n$ is a constant, if we define $\tilde{K}_i=\frac{i\mu^+_n-R_i}{n-i}$, then $\tilde{K}_i$ is also an $\mathcal{F}_i-$martingale. 
It follows that for any $\epsilon>0$,
\begin{eqnarray}\label{eqn:prob bound on deviation of R}
\p{\frac{1}{c_n}\max\limits_{i\le s}|i\mu^+_n-R_i|>\epsilon}&\le& \frac{n^2}{\epsilon^2c^2_n}\E{\left(\max\limits_{i\le s}\frac{|i\mu^+_n-R_i|}{n-i}\right)^2}\nonumber\\
&\le& \frac{4n^2\E{(s\mu^+_n-R_s)^2}}{\epsilon^2c^2_n(n-s)^2},
\end{eqnarray}
where in the first line we use Markov's inequality and in the last line we use the $L^2$ maximal inequality for martingales (see, e.g. Theorem 5.4.3 in \cite{Durrett2010}). 

Since the process $(R_s, 0\le s\le n)$ has exchangeable increments, we have $\e{R_s}=s\mu^+_n$.
Let $R^\ast_s=\sum\limits_{i\le s} J_i$ where $J_1, \ldots, J_s$ are i.i.d. random variables with $J_1\overset{d}{=}R_1$. Then Proposition \ref{prop:dilution} gives 
\begin{eqnarray*}
\E{R^2_s}\le \E{{R^\ast_s}^2}=\E{(J_1+\cdots+J_s)^2} &=& s\E{J^2_1}+s(s-1)(\e{J_1})^2\\
&=& s({\sigma^+_n}^2-\mu^+_n)+s(s-1){\mu^+_n}^2
\end{eqnarray*}
Therefore,
\begin{equation*}
\E{(s\mu^+_n-R_s)^2}=\E{R^2_s}-s^2{\mu^+_n}^2\le s({\sigma^+_n}^2-\mu^+_n)-s{\mu^+_n}^2\le s{\sigma^+_n}^2.
\end{equation*}
Now take $s=s(n)=c^2_n/\epsilon$ in (\ref{eqn:prob bound on deviation of R}). For $n$ large this is less than $n/2$, so  $(n-s)^2>n^2/4$ and we obtain
\[\p{\frac{1}{c_n}\max\limits_{i\le c^2_n/\epsilon}|i\mu^+_n-R_i|>\epsilon}\le \frac{16s{\sigma^+_n}^2}{\epsilon^2c^2_n}=\frac{16{\sigma^+_n}^2}{\epsilon^3}\le \epsilon,\]
the last inequality holding for $n$ large since $\sigma^+_n\to 0$ as $n\to \infty$. This completes the proof.
\end{proof}

Recall that in Section \ref{sec:intro} we let $\tau(x)= \inf(t: B(t) \le -x)$ for $x \ge 0$. By (\ref{eqn:size of small trees}) if we let $\tau_n=\sum_{1 \le i < c_n} |T_{n,i}|=n-|T_{n,c_n}|$ be the total size of non-marked trees of $(F_n^*,v_n)$, then since $S_n$ is the coding process of $(F_n^*,v_n)$ we have 
\[\tau_n=\inf \{k: S_{n,k}=-(c_n-1)\}.\]
From this we immediately get the following corollary of Theorem \ref{thm:walk convergence}.

\begin{cor}\label{cor:size of forests of small trees}
Given the assumptions in Theorem \ref{thm:main}, we have
\begin{equation}\label{eqn:small trees size}
\frac{\tau_n}{c^2_n}\overset{d}{\to}\tau(\frac{1}{\sigma}),
\end{equation}
where $(B(t), t\ge 0)$ is standard Brownian Motion.
\end{cor}

\noindent{\bf Remark.} Note that the right-hand side of (\ref{eqn:small trees size}) has density 
$\frac{1}{\sigma\sqrt{2\pi t^3}}\exp\left(-\frac{1}{2t\sigma^2}\right)dt$\,;
see, e.g., Theorem 6.9 in \cite{SchillingPartzsch2012}.

The corollary above in fact tells us something about the size of the largest tree $T^\downarrow_{n,1}$.
\begin{cor}\label{cor:marked tree is largest tree}
For a marked cyclic forest $(F,v)$, let $MT(F,v)$ denoted the marked tree, i.e. the tree of $F$ containing $v$. Then 
\[\p{MT(F_n^*,v_n)= T^\downarrow_{n,1}}\to 1\]
as $n\to\infty$.
\end{cor}
\begin{proof}
It is clear that
\begin{eqnarray*}
\p{MT(F_n^*,v_n)\ne T^\downarrow_{n,1}} &\le& \p{|MT(F_n^*,v_n)|<n/2}\\
&=& \p{\tau_n>n/2}=\p{\frac{\tau_n}{c^2_n}>\frac{n}{2c^2_n}}\to 0
\end{eqnarray*}
where in the last line, the first equation is by Lemma \ref{lem:maps forest to marked cyclic forest} and the final convergence is by Corollary \ref{cor:size of forests of small trees} and the assumption $c^2_n=o(n)$.
\end{proof}

Now we are ready to prove Corollary \ref{cor: small tree sizes}.
\begin{proof}[Proof of Corollary \ref{cor: small tree sizes}]
As noted, it suffices to prove (\ref{eqn:sum of small tree sizes}) and (\ref{eqn:tree size convergence}). Corollary \ref{cor:size of forests of small trees} and Corollary \ref{cor:marked tree is largest tree} together imply (\ref{eqn:sum of small tree sizes}). 

For (\ref{eqn:tree size convergence}), 
first note that by Lemma \ref{lem:bijection}, the process $S_n=(S_{n,k},0 \le k \le n)$ has the same law as the coding walk $W(F_n)$ of $F_n$. Applying Corollary \ref{cor:marked tree is largest tree} then yields that the law of $(|T_{n,2}^{\downarrow}|,\ldots,|T_{n,j}^{\downarrow}|)$ is asymptotically equivalent to the law of $(g^n_1-d^n_1,\ldots,g^n_{j-1}-d^n_{j-1})$, the first $j-1$ ranked excursion lengths of $S_n$ above its running minimum before time $\tau_n$. Using this equivalence, (\ref{eqn:tree size convergence}) now follows from Theorem \ref{thm:walk convergence} by the Portmanteau Theorem (\cite{MortersPeresbook}, Theorem 12.6), since the vector $(g_1-d_1,\ldots,g_{j-1}-d_{j-1})$ has a density.
\end{proof}

\section{\bf Empirical degree sequences of trees}\label{sec:verify degree conditions}

In this section we aim to prove (\ref{eqn:degree proportion converge}) and (\ref{eqn: sigma square converge}). 
%That is,
%\[\mbox{for any fixed } i, l\in\mathbb{N}, l\ge 2,\ p^i_{n,l}-p^i_n\overset{p}{\to}0, \mbox{ as } n\to\infty,\]
%\[
%\mbox{for any } l\ge 2, \ \sigma^2(p_{n,l})-\sigma^2(p_n)\overset{p}{\to} 0, \mbox{ as } n\to\infty.\]

For $i\ge 0$ and $x\le n$, let
\[Q^i_n(x)\!:=|\{1\le j\le x: C_{n,j}=i\}|\] where $(C_{n,1}, \cdots, C_{n,n})$ is a uniformly random permutation of $d(s_n)$ and that $S_{n,k} = \sum_{j=1}^k (C_{n,j}-1)$. Let $\mathcal{F}_j=\sigma(C_{n,1},\ldots, C_{n,j})$. Since $Q^i_n(n)=s^i_n=np^i_n$ and the process $(Q^i_n(j), 0\le j\le n)$ has exchangeable increments,
\[\E{Q^i_n(j+1)~|~\mathcal{F}_j}=Q^i_n(j)+\frac{np^i_n-Q^i_n(j)}{n-j}.\]
Setting $K_j=\frac{np^i_n-Q^i_n(j)}{n-j}$ for $0\le i\le n$, then \[\E{K_{j+1}~|~\mathcal{F}_j}=\frac{np^i_n-Q^i_n(j)}{n-(j+1)}-\frac{np^i_n-Q^i_n(j)}{(n-(j+1))(n-j)}=K_j,\] so $K_j$ is an $\mathcal{F}_j-$martingale. If we let $\tilde{K}_j=\frac{jp^i_n-Q^i_n(j)}{n-j}$, then $\tilde{K}_j=K_j-p^i_n$, so $\tilde{K}_j$ is also an $\mathcal{F}_j-$martingale. 

%% The following part is commented our as it is an unsuccessful try of using L^2 maiximal inequality for martingales to prove the tree degree proportional convergence
%Similar to how we obtain (\ref{eqn:prob bound on deviation of R}), by applying Markov's inequality and the $L^2$ maximal inequality for martingales (see, e.g. Theorem 5.4.3 in \cite{Durrett2010}), we have that for any $\epsilon>0$
%\begin{eqnarray}\label{eqn:apply markov+maximal martingale ineq}
%\p{\exists m> c_n:\left|\frac{Q^i_n(m)}{m}-p^i_n\right|>\epsilon}&\le& \p{\max\limits_{j\le n-c_n}\left|\frac{Q^i_n(j)-jp^i_n}{n-j}\right|>\epsilon} \nonumber \\
%&\le&\frac{1}{\epsilon^2}\E{\left(\max\limits_{j\le n-c_n}\left|\frac{Q^i_n(j)-jp^i_n}{n-j}\right|\right)^2}\nonumber \\
%&\le& \frac{4}{\epsilon^2c^2_n}\E{(Q^i_n(n-c_n)-(n-c_n)p^i_n)^2}.
%\end{eqnarray}
%
%Note that for any $1\le x\le n, \e{Q^i_n(x)}=xp^i_n$ and by Proposition \ref{prop:dilution}
%\begin{equation}\label{eqn: apply dilution prop}
%\E{(Q^i_n(x))^2}\le \E{(\tilde{Q}^i_n(x))^2}=xp^i_n(1-p^i_n)+(xp^i_n)^2
%\end{equation} 
%where $\tilde{Q}^i_n(x)$ is a random variable with Binomial distribution $\mathrm{Bin}(x, p^i_n)$.
%
%Combining (\ref{eqn:apply markov+maximal martingale ineq}) and (\ref{eqn: apply dilution prop}) (with $x=n-c_n$), we get
%\[\p{\exists m>c_n: \left|\frac{Q^i_n(m)}{m}-p^i_n\right|>\epsilon}\le \frac{4}{\epsilon^2c^2_n}(n-c_n)p^i_n(1-p^i_n).\]
%
%This bound does not seem very useful since $c^2_n=o(n)$... So let's go back to using McDiarmid's concentration result. We still use the notation we just introduced in this section and 

We now use the following martingale bound from \cite{McDiarmid1998}. Let $\{X_j\}_{j=0}^n$ be a bounded martingale adapted to a filtration $\{\mathcal{F}_j\}_{j=0}^n$. Let $V=\sum\limits_{j=0}^{n-1}var\{X_{j+1}~|~\mathcal{F}_j\},$ where
$$var\{X_{j+1}~|~\mathcal{F}_j\}:=\E{(X_{j+1}-X_j)^2~|~\mathcal{F}_j}=\E{X_{j+1}^2~|~\mathcal{F}_j}-X_j^2.$$ Let $$v=\mbox{ess sup }V, \mbox{ and }b=\max\limits_{0\le j\le n-1}\mbox{ess sup}(X_{j+1}-X_j~|~\mathcal{F}_j).$$

\begin{thm}[\cite{McDiarmid1998}, Theorem 3.15]\label{thm:concentration}
For any $t\ge 0$,
$$\p{\max\limits_{0\le j\le n} X_j\ge t}\le \exp \left(-\frac{t^2}{2v(1+bt\backslash(3v))}\right).$$
\end{thm}
We shall apply this theorem to bound the fluctuations of $Q^i_n(s)$.
\begin{prop}\label{prop: martingale_concentration}
For any $0<t<1$, we have
\begin{equation}%\label{eqn:concentration for starting interval less than x}
\p{\exists s>c_n: \left|p^i_n-\frac{Q^i_n(s)}{s}\right |\ge t}\le\exp\left(-\frac{3t^2c_n}{5}\right).
\end{equation} 
%In particular, for $\sqrt{\frac{5}{3}}c^{-1/3}_n<t<1$,
%\[\p{\exists s>c_n: |p^i_n-\frac{Q^i_n(s)}{s}|\ge t}\le\exp\left(-c^{1/3}_n\right).\]
\end{prop}
\begin{proof}
It is not hard to show that for any $0\le j\le n-2$, \[var\{\tilde{K}_{j+1}~|~\mathcal{F}_j\}\le\frac{1}{4}\cdot\frac{1}{(n-(j+1))^2};\] see, e.g., Lemma 3.2 of \cite{Lei2017+}. Thus, for $1\le x\le n-2$,
\begin{eqnarray*}
V=\sum\limits_{j=0}^{x-1}var\{\tilde{K}_{j+1}~|~\mathcal{F}_j\}&\le&\frac{1}{4}
\sum_{j=0}^{x-1}\frac{1}{(n-(j+1))^2}\\
&\le&\frac{1}{4}\int^{n-1}_{n-x-1}\frac{1}{m^2}\mathrm{d}m=\frac{x}{4(n-1)(n-x-1)}.
\end{eqnarray*}
On the other hand, for $0\le j\le x-1$, if $Q^i_n(j+1)=Q^i_n(j)$, then \[|\tilde{K}_{j+1}-\tilde{K}_j|=\left|\frac{np^i_n-Q^i_n(j)}{(n-(j+1))(n-j)}\right|\le \frac{1}{n-x},\]
while if $Q^i_n(j+1)=Q^i_n(j)+1$, then \[|\tilde{K}_{j+1}-\tilde{K}_j|=\left|\frac{np^i_n-Q^i_n(j)}{(n-(j+1))(n-j)}-\frac{1}{n-(j+1)}\right|\le\frac{1}{n-x}.\]
Applying Theorem \ref{thm:concentration} to both $\{\tilde{K}_j\}^x_{j=0}$ and $\{-\tilde{K}_j\}^x_{j=0}$ with 
$x=n-c_n$, we have
\[v\le \frac{1}{2c_n},~ b\le \frac{1}{c_n}.\]
Hence, for $t \le 1$, 
\[\p{\max\limits_{0\le j\le n-c_n}\left|p^i_n-\frac{np^i_n-Q^i_n(j)}{n-j}\right|\ge t}\le\exp\left(-\frac{t^2}{\frac{1}{c_n}+\frac{2t}{3c_n}}\right)\le\exp\left(-\frac{3t^2c_n}{5}\right).\]
Using the exchangeability of $C_{n,1},\ldots,C_{n,n}$, it follows that
\begin{eqnarray*}
\p{\exists s>c_n: |p^i_n-\frac{Q^i_n(s)}{s}|\ge t}&=&\p{\max\limits_{0\le j\le n-c_n}\left|p^i_n-\frac{np^i_n-Q^i_n(j)}{n-j}\right|\ge t}\\
&\le&\exp\left(-\frac{3t^2c_n}{5}\right).
\end{eqnarray*}
\end{proof}

We next give the proofs of (\ref{eqn:degree proportion converge}) and (\ref{eqn: sigma square converge}). 
In both proofs we use the coupling between $F_n$, $(F_n^*,v_n)$ and $S_n$ explained at the end of Section \ref{sec:combinatorics}. 

\begin{proof}[Proof of (\ref{eqn:degree proportion converge})]
Fix $i \ge 0$ and $l \ge 2$. By Corollary \ref{cor:marked tree is largest tree}, with high probability $T^{\downarrow}_{n,1}=T_{n,c_n}$, i.e., $T^{\downarrow}_{n,1}$ is the last tree of $(F_n^*,v_n)$, in which case $T_{n,l}^{\downarrow}=T_{n,j}$ for some $j < c_n$. Recall that $\tau_n=\sum_{1 \le k < c_n} |T_{n,k}|$. 

Let $1 \le j < c_n$, and suppose $|\{v \in T_{n,j}: k(v)=i\}|/|T_{n,j}| \not \in [p_n^i-\delta,p_n^i+\delta]$. Suppose that $|T_{n,j}| > \delta c_n^2 > c_n$ and $\tau_n<c^3_n$. Then there must exist $m>c_n$ and $1\le u\le \tau_n-m$ such that \[\left|\frac{\left|\{t\in[m]: C_{n,u+t}=i\}\right|}{m}-p^i_n\right|>\delta.\] By union bound and the exchangeability of $(C_{n,1},\ldots, C_{n,n})$, the probability of this is bounded above by $\tau_n\p{\exists m>c_n:\left|\frac{Q^i_n(m)}{m}-p^i_n\right|>\delta}$.
%We now use the result of the following simple computation. Let $(x_t,1 \le t \le r+s) \in \{0,1\}^{r+s}$ and for $0 \le t \le r+s$ let $y_t = x_1+\ldots+x_t$. Write $y_{r+s}=\alpha(r+s)$. For any $\epsilon \in \R$, if $y_{r+s}-y_r = (\alpha+\epsilon)s$ then $y_r = \alpha r-\epsilon s = (\alpha-\epsilon s/r)r$. 
%
%The identity of the preceding paragraph implies the following. Let $1 \le j < c_n$, and suppose $|\{v \in T_{n,j}: k(v)=i\}|/|T_{n,j}| \not \in [p_n^i-\delta,p_n^i+\delta]$. 
%Suppose that $|T_{n,j}| > \delta c_n^2 > c_n$ and $\sum_{1 \le k < c_n} |T_{n,k}| \le c_n^2/\delta$. Then, considering the values of the process $(Q_n^i(s),0 \le s \le n)$ just before and just after exploring $T_{n,j}$, it follows that 
%\[
%\max_{s > c_n} \left| p_n^i - \frac{Q_n^i(s)}{s} \right| 
%\ge \frac{1}{2} \cdot \delta \cdot \frac{|T_{n,j}|}{\tau_n} \ge \frac{\delta^3}{2}\, . 
%\]
Thus, for $l \ge 2$, for $n$ large enough that $\delta c_n^2 > c_n$, we have 
\begin{align*}
\p{\left| p_{n,l}^i - p_n^i \right| > \delta} & 
\le 
\p{\tau_n > c_n^3} + \p{|T^{\downarrow}_{n,l}| < \delta c_n^2} 
\\
& 
+ \p{T^{\downarrow}_{n,1}\ne T_{n,c_n}} + c^3_n
\p{\max_{s > c_n} \left| p_n^i - \frac{Q_n^i(s)}{s} \right| > \delta}
\end{align*}
For any $\epsilon > 0, \p{\tau_n > c_n^3} < \epsilon/3$ by Corollary \ref{cor:size of forests of small trees} for $n$ large enough, and 
$\p{|T^{\downarrow}_{n,l}| < \delta c_n^2}< \epsilon/3$ by Corollary \ref{cor: small tree sizes}. The second last probability tends to zero by Corollary \ref{cor:marked tree is largest tree}. And for the last probability, for $n$ large enough, $\sqrt{\frac{5}{3}}c^{-1/3}_n<\delta$, hence Proposition \ref{prop: martingale_concentration} gives upper bound $c^3_n\exp\left(-c^{1/3}_n\right)$, which tends to zero. Thus, $\p{\left| p_{n,l}^i - p_n^i \right| > \delta} < \epsilon$ for $n$ large; this proves (\ref{eqn:degree proportion converge}) for $i \ge 0$ and $l > 1$.

Finally, since $T_{n,1}^{\downarrow}/n \to 1$, the fact that $\left| p_{n,1}^i - p_n^i \right| \to 0$ in probability for each $i\ge 0$ is immediate.
\end{proof}
\begin{proof}[Proof of (\ref{eqn: sigma square converge})]
Fix $\epsilon>0$. By Corollary \ref{cor:size of forests of small trees}, we can pick $M>0$ large enough such that for $n$ large enough,
\begin{equation}\label{eqn:1_epsilon_term}
\p{\tau_n>Mc^2_n}<\epsilon. 
\end{equation} 
By Corollary \ref{cor:marked tree is largest tree}, we have that for $n$ large enough, 
\begin{equation}\label{eqn:2_epsilon_term}
\p{T_{n,c_n}\ne T^\downarrow_{n,1}}< \epsilon\, .
\end{equation}
For this $\epsilon>0$, there exists $\delta>0$ such that $\p{g_{l-1}-d_{l-1} \le\delta}<\epsilon/2$, so by Corollary \ref{cor: small tree sizes}, for $n$ large
\begin{equation}\label{eqn:3_epsilon_term}
\p{\frac{|T^\downarrow_{n,l}|}{c^2_n}\le\delta}<\epsilon.
\end{equation}
Next we fix $t>0$ large enough such that 
\begin{equation}\label{eqn: larger than t moments small for pi}
\E{C^2_{n,1}\mathbbm{1}_{C_{n,1}\ge t}}<\frac{\epsilon^2\delta}{M} \mbox{  and  } \sum\limits_{i>t}i^2p^i_n<\epsilon\, ;
\end{equation}
this is possible since $p_n=(p^i_n, i\ge 0)\to p=(p^i, i\ge 0)$ in $L^2$. 
For fixed $l\ge 2$ we have
\begin{eqnarray}\label{eqn:intermediate bounding second moment}
|\sigma^2(p_{n,l})-\sigma^2(p_n)|&\le&|\sum\limits_{i\le t} i^2(p^i_{n,l}-p^i_n)|+\sum\limits_{i>t}i^2p^i_n+\sum\limits_{i>t}i^2p^i_{n,l}\nonumber\\
&\le& |\sum\limits_{i\le t} i^2(p^i_{n,l}-p^i_n)|+\epsilon+\sum\limits_{i>t}i^2p^i_{n,l}
\end{eqnarray}
where we use (\ref{eqn: larger than t moments small for pi}) in the second line. 

Let $L_n=\sum_{j\le Mc_n^2}C^2_{n,j}\mathbbm{1}_{C_{n,j}\ge t}$. 
If $T_{n,c_n}=T^\downarrow_{n,1}$ 
and $\tau_n \le Mc^2_n$ 
then $\sum\limits_{i>t}i^2p^i_{n,l} \le {L_n}/|T^\downarrow_{n,l}|$.
Hence 
\begin{eqnarray}\label{eqn: almost there bounding}
&&\p{|\sigma^2(p_{n,l})-\sigma^2(p_n)|\ge 3\epsilon}\nonumber\\
&\le&\p{|\sigma^2(p_{n,l})-\sigma^2(p_n)|\ge 3\epsilon, \tau_n\le Mc^2_n, T_{n,c_n}=T^\downarrow_{n,1}, \frac{|T^\downarrow_{n,l}|}{c^2_n}>\delta}\nonumber\\
&+&\p{\tau_n>Mc^2_n}+\p{T_{n,c_n}\neq T^\downarrow_{n,1}}+\p{\frac{|T^\downarrow_{n,l}|}{c^2_n}\le\delta}\nonumber\\
&\le& \p{|\sum\limits_{i\le t}i^2(p^i_{n,l}-p^i_n)|\ge\epsilon}+\p{\frac{L_n}{|T^\downarrow_{n,l}|}>\epsilon, \frac{|T^\downarrow_{n,l}|}{c^2_n}>\delta}+3\epsilon
\end{eqnarray}
where we use (\ref{eqn:1_epsilon_term}), (\ref{eqn:2_epsilon_term}), (\ref{eqn:3_epsilon_term}), (\ref{eqn:intermediate bounding second moment}) and the aforementioned stochastic dominance in the last line.

Since $t$ is fixed, we can use (\ref{eqn:degree proportion converge}) to conclude that the first summand of (\ref{eqn: almost there bounding}) can be made arbitrarily small by taking $n$ large enough. For the second summand, 
note that by exchangeability and (\ref{eqn: larger than t moments small for pi}),
\[
\e{L_n}=Mc^2_n\E{C^2_{n,1}\mathbbm{1}_{C_{n,1}\ge t}}<c^2_n\epsilon^2\delta\, ,
\]
so 
\[\p{\frac{L_n}{|T^\downarrow_{n,l}|}>\epsilon, \frac{|T^\downarrow_{n,l}|}{c^2_n}>\delta}\le\p{\frac{L_n}{c^2_n}>\epsilon\delta}\le\frac{\E{\frac{L_n}{c^2_n}}}{\epsilon\delta}<\epsilon.
\]
This completes the proof of (\ref{eqn: sigma square converge}) for $l\ge 2$. Again since $T^\downarrow_{n,1}/n\to 1$, (\ref{eqn: sigma square converge}) is immediate for $l=1$ case.
\end{proof}

%. The times $\{\tau(x), 0 \le x \le 1\}$ are precisely the times $\{t \le \tau(1): R(t)=0\}$. 
%
%The set $[0,\tau(1)]-

% and let $\tau = \inf(t \ge 0: B(t)=-1)$. Then $B'=(B(t), 0 \le t \le \tau)$
% is the contour process of a continuum random forest $\cF_{B'}=(\cT_i,i \ge 0)$
%the forest coded by a linear Brownian motion stopped when it first visits
%A definition: \nomenclature[Xifu]{$\xi$-fu}{A $\xi$-fu: a $d_{\mathrm{GH}}$-bar}
%A {\em $\xi$-fu} is a {\em $\mu$-bar}. 

%%%%%%%%%%%%%%%%%%%%%%%%%%%%
% CONCLUSION
%%%%%%%%%%%%%%%%%%%%%%%%%%%%
%\section{Conclusion}\label{sec:conc}

%%%%%%%%%%%%%%%%%%%%%%%%%%%%
% ACKNOWLEDGEMENTS
%%%%%%%%%%%%%%%%%%%%%%%%%%%%
\addtocontents{toc}{\SkipTocEntry} %Suppresses  acknowledgements from Table of Contents
\section{\bf Acknowledgements}
I am grateful for my supervisor Prof. Louigi Addario-Berry, who suggested the project and gave numerous helpful insights and suggestions during our discussions.

%%%%%%%%%%%%%%%%%%%%%%%%%%%%
% TABLE OF CONTENTS
%%%%%%%%%%%%%%%%%%%%%%%%%%%%
\small
\addtocontents{toc}{\SkipTocEntry} %Suppresses list of notation from Table of Contents
\printnomenclature[3.1cm]
\normalsize

%%%%%%%%%%%%%%%%%%%%%%%%%%%%
% BIBLIOGRAPHY
%%%%%%%%%%%%%%%%%%%%%%%%%%%%

%%Use these commands if you're using a bib file called template.bib 
\small 
\bibliographystyle{plainnat}
\bibliography{rf_basic}
\normalsize

%Use this for bibliography within the main .tex file
%\begin{thebibliography}{34}
%\providecommand{\natexlab}[1]{#1}
%\providecommand{\url}[1]{\texttt{#1}}
%\expandafter\ifx\csname urlstyle\endcsname\relax
%  \providecommand{\doi}[1]{doi: #1}\else
%  \providecommand{\doi}{doi: \begingroup \urlstyle{rm}\Url}\fi
%  
%\bibitem[Schaeffer(1998)]{Schaeffer1998Conjugaison}
%Gilles Schaeffer.
%\newblock \emph{Conjugaison d'arbres et cartes combinatoires al{\'e}atoires}.
%\newblock PhD thesis, Universit{\'e} Bordeaux~I, 1998.
%
%\end{thebibliography}
%\normalsize

%%%%%%%%%%%%%%%%%%%%%%%%%%%%
% APPENDIX
%%%%%%%%%%%%%%%%%%%%%%%%%%%%
%\appendix
%
%\section{A section}

\end{document}